\documentclass[reqno]{amsart}
\usepackage{amsmath}

\usepackage{amsfonts}
\usepackage{amssymb}
\usepackage{hyperref}

\begin{document}
\title[ Positive solutions ]{Positive solutions of a nonlinear three-point boundary value problem with integral boundary conditions}
\author[F. Haddouchi]{Faouzi Haddouchi}
\address{Faouzi Haddouchi\\
Department of Physics, University of Sciences and Technology of
Oran, El Mnaouar, BP 1505, 31000 Oran, Algeria}
\email{haddouch@univ-usto.dz}

\begin{abstract}
In this paper, by using Krasnoselskii's fixed point theorem in a cone, we study the
existence of single and multiple positive solutions to the three-point integral boundary value
problem (BVP)
\begin{equation*} \label{eq-1}
\begin{gathered}
{u^{\prime \prime }}(t)+a(t)f(u(t))=0,\ \ 0<t<T, \\
u^{\prime}(0)=0, \ u(T)={\alpha}\int_{0}^{\eta}u(s)ds,
\end{gathered}
\end{equation*}
where $0<{\eta}<T$, $0<{\alpha}< \frac{1}{{\eta}}$, are given constants.
\end{abstract}

\maketitle \numberwithin{equation}{section}
\newtheorem{theorem}{Theorem}[section]
\newtheorem{lemma}[theorem]{Lemma} \newtheorem{proposition}[theorem]{%
Proposition} \newtheorem{corollary}[theorem]{Corollary} \newtheorem{remark}[%
theorem]{Remark}
\newtheorem{exmp}{Example}[section]

\section{Introduction}
In this work, we study the existence of positive solutions of the following three-point integral
boundary value problem (BVP):

\begin{equation} \label{eq-2}
{u^{\prime \prime }}(t)+a(t)f(u(t))=0,\ t\in(0,T),
\end{equation}
\begin{equation} \label{eq-3}
u^{\prime}(0)=0, \ u(T)={\alpha}\int_{0}^{\eta}u(s)ds,
\end{equation}
where $0<{\eta}<T$ and $0<{\alpha}<\frac{1}{{\eta}}$, and
\begin{itemize}
\item[(B1)] $f\in C([0,\infty),[0,\infty))$;
\item[(B2)] $a\in C([0,T],[0,\infty))$ and there exists $t_{0}\in[0,T]$ such that
$a(t_{0})>0$.
\end{itemize}

Set
\begin{equation} \label{eq-4}
f_{0}=\lim_{u\rightarrow0^{+}}\frac{f(u)}{u}, \
f_{\infty}=\lim_{u\rightarrow\infty}\frac{f(u)}{u}.
\end{equation}

The study of the existence of solutions of multi-point boundary value
problems for linear second-order ordinary differential equations was
initiated by II'in and Moiseev \cite{Ilin}.
Then Gupta \cite{Gupt} studied three-point boundary value problems for nonlinear
second-order ordinary differential equations. Since then, the existence of positive
solutions for nonlinear second order three-point boundary-value problems has been
studied by many authors by using the fixed point theorem, nonlinear alternative of the Leray-Schauder
approach, or coincidence degree theory. We refer the reader to
\cite{Cheng,Guo,Han,Li,Pang,Sun1,Sun2,Xu,Ander1,Ander2,He,Ma1,Ma2,Ma3,Ma4,Ma5,Feng,Maran,Luo, Liang1, Liang2,
Liu1,Liu2,Liu3,Tarib,Webb1,Webb2,Webb3,Webb4,Webb5,Webb6,Webb7} and the references therein.

Liu \cite{Liu4} proved the existence of single and multiple positive solutions for the three-point integral
boundary value problem (BVP)
\begin{equation} \label{eq-5}
{u^{\prime \prime }}(t)+a(t)f(u(t))=0,\ t\in(0,1),
\end{equation}
\begin{equation} \label{eq-6}
u^{\prime}(0)=0, \ u(1)={\beta}u(\eta),
\end{equation}
where $0<{\eta}<1$ and $0<{\beta}<1$.

In \cite{Haddou1, Haddou2}, the authors considered the following three-point integral boundary value problem
(BVP)

\begin{equation} \label{eq-7}
{u^{\prime \prime }}(t)+a(t)f(u(t))=0,\ t\in(0,T),
\end{equation}
\begin{equation} \label{eq-8}
u(0)={\beta}u(\eta),\ u(T)={\alpha}\int_{0}^{\eta}u(s)ds,
\end{equation}
where $0<{\eta}<T$ and $0<{\alpha}<\frac{2T}{{\eta}^{2}}$,
$0\leq{\beta}<\frac{2T-\alpha\eta^{2}}{\alpha\eta^{2}-2\eta+2T}$,  $f\in C([0,\infty),[0,\infty))$, $a\in
C([0,T],[0,\infty))$ and there exists $t_{0}\in[\eta,T]$ such that $a(t_{0})>0$. They obtained the existence of
single and multiple positive solutions by using the Krasnoselskii's  fixed point theorem.

Motivated by the results of \cite{Liu4, Haddou1, Haddou2} the aim of this paper is to establish some simple
criterions for the existence of positive solutions of the BVP \eqref{eq-2},\eqref{eq-3}.
In Section 2, we prove several preliminary results that will be used to prove our results. In Section 3, we
discuss the existence of single positive solution for the BVP \eqref{eq-2},\eqref{eq-3} under $f_{0}=0$,
$f_{\infty}=\infty$ or $f_{0}=\infty$, $f_{\infty}=0$. In Section 4, we establish the existence conditions of
two positive solutions for the BVP \eqref{eq-2},\eqref{eq-3} under $f_{0}=f_{\infty}=\infty$ or
$f_{0}=f_{\infty}=0$. In Section 5, we also obtain some existence results for positive solutions of the BVP
\eqref{eq-2},\eqref{eq-3} under $f_{0}, f_{\infty}\not\in \left\{0,\infty\right\}$. Finally, in Section 6, we
give some examples to illustrate our results.

The key tool in our approach is the
following Krasnoselskii's fixed point theorem in a cone \cite{Krasn}.

\begin{theorem}\label{theo 2.1}\cite{Krasn}.
Let $E$ be a Banach space, and let $K\subset E$ be a cone. Assume
$\Omega_{1}$, $\Omega _{2}$ are open bounded subsets of $E$ with $0\in \Omega _{1}$%
, $\overline{\Omega }_{1}\subset \Omega _{2}$, and let

\begin{equation*}\label{eq-11}
A: K\cap (\overline{\Omega }_{2}\backslash  \Omega
_{1})\longrightarrow K
\end{equation*}

be a completely continuous operator such that either

(i) $\ \left\Vert Au\right\Vert \leq \left\Vert u\right\Vert $, $\
u\in K\cap \partial \Omega _{1}$, \ and $\left\Vert Au\right\Vert
\geq \left\Vert u\right\Vert $, $\ u\in K\cap \partial \Omega _{2}$;
or

(ii) $\left\Vert Au\right\Vert \geq \left\Vert u\right\Vert $, $\
u\in K\cap
\partial \Omega _{1}$, \ and $\left\Vert Au\right\Vert \leq \left\Vert
u\right\Vert $, $\ u\in K\cap \partial \Omega _{2}$

hold. Then $A$ has a fixed point in $K\cap
(\overline{\Omega}_{2}\backslash $ $\Omega _{1})$.
\end{theorem}

\section{Preliminaries}
To prove the main existence results we will employ several
straightforward lemmas.

\begin{lemma}\label{lem 2.1}
Let $\alpha\eta\neq1$. Then for $y\in C([0,T],\mathbb{R})$, the problem
\begin{equation}\label{eq-9}
{u^{\prime \prime }}(t)+y(t)=0, \ t\in (0,T),
\end{equation}
\begin{equation}\label{eq-10}
{u^{\prime}}(0)=0, \ u(T)=\alpha \int_{0}^{\eta }u(s)ds
\end{equation}
has a unique solution
\begin{eqnarray*}
u(t)&=&\frac{1}{1-\alpha\eta}\int_{0}^{T}(T-s)y(s)ds-\frac{\alpha}{2(1-\alpha\eta)}\int_{0}^{\eta }(\eta
-s)^{2}y(s)ds \\
&&-\int_{0}^{t}(t-s)y(s)ds.
\end{eqnarray*}
\end{lemma}
\begin{proof}
From \eqref{eq-9}, we have
\begin{equation}\label{eq-11}
u(t)=u(0)-\int_{0}^{t}(t-s)y(s)ds
\end{equation}
Integrating \eqref{eq-11} from $0$ to $\eta $, where $\eta \in
(0,T)$, we have
\begin{eqnarray*}
\int_{0}^{\eta }u(s)ds=u(0){\eta}-\frac{1}{2}\int_{0}^{\eta }(\eta -s)^{2}y(s)ds.
\end{eqnarray*}
Since
\begin{equation*}\label{eq-12}
u(T)=u(0)-\int_{0}^{T}(T-s)y(s)ds.
\end{equation*}
From $\ u(T)=\alpha \int_{0}^{\eta }u(s)ds$, we have
\begin{equation*}\label{eq-13}
(1-\alpha \eta )u(0)=\int_{0}^{T}(T-s)y(s)ds-\frac{\alpha }{2}\int_{0}^{\eta }(\eta
-s)^{2}y(s)ds.
\end{equation*}
Therefore,
\begin{eqnarray*}
u(0)=\frac{1}{1-\alpha \eta}\int_{0}^{T}(T-s)y(s)ds -\frac{\alpha} {2(1-\alpha \eta)}\int_{0}^{\eta }(\eta
-s)^{2}y(s)ds,
\end{eqnarray*}
from which it follows that
\begin{eqnarray*}
u(t)&=&\frac{1}{1-\alpha\eta}\int_{0}^{T}(T-s)y(s)ds-\frac{\alpha}{2(1-\alpha\eta)}\int_{0}^{\eta }(\eta
-s)^{2}y(s)ds \\
&&-\int_{0}^{t}(t-s)y(s)ds.
\end{eqnarray*}
\end{proof}

\begin{lemma}\label{lem 2.2}
Let $0<\alpha <\frac{1}{\eta}$. If $y\in C([0,T],[0,\infty
))$, then the unique solution $u$ of
\eqref{eq-9}-\eqref{eq-10} satisfies $\ u(t)\geq 0$ for $t\in [0,T]$.
\end{lemma}
\begin{proof}
From the fact that ${u^{\prime \prime }}(t)=-y(t)\leq0$ , we know that the graph of $u(t)$ is
concave down on $\left(0,T\right)$ and ${u^{\prime}}(t)$ monotone decreasing. Thus ${u^{\prime}}(t)\leq
{u^{\prime}}(0)=0$
and $u(t)$ is a monotone decreasing function, this is $u(t)\geq u(T)$ ($t\in\left[0,T\right]$). So, if
$u(T)\geq0$, then $u(t)\geq0$ for $t\in\left[0,T\right]$.

If $u(T)<0$, then $\int_{0}^{\eta }u(s)ds<0$.
Since the graph of $u$ is concave down, we get
\begin{equation} \label{eq-14}
\int_{0}^{\eta }u(s)ds\geq \frac{\eta }{2}(u(0)+u(\eta )).
\end{equation}
It implies
\begin{equation} \label{eq-15}
u(T)=\alpha \int_{0}^{\eta }u(s)ds>\frac{1}{\eta}\int_{0}^{\eta }u(s)ds\geq\frac{u(0)+u(\eta)}{2}.
\end{equation}
Since $u(t)$ is monotone decreasing, by \eqref{eq-15}, we obtain
\begin{equation} \label{eq-16}
u(T)>u(\eta),
\end{equation}
which contradicts the fact that $u(t)$ is a monotone decreasing function.
\end{proof}

\begin{lemma}\label{lem 2.3}
Let $\ \alpha >\frac{1}{\eta}$. If $y\in C([0,T],[0,\infty
))$, then the problem \eqref{eq-9}-\eqref{eq-10} has no positive solution.
\end{lemma}
\begin{proof}
Suppose that problem \eqref{eq-9}-\eqref{eq-10} has a positive
solution $u$ satisfying $u(t)\geq 0,$ $t\in [0,T]$.

If $u(T)>0$, then $\int_{0}^{\eta }u(s)ds>0$. It implies
\begin{equation*} \label{eq-17}
u(T)=\alpha\int_{0}^{\eta }u(s)ds > \frac{u(0)+u(\eta)}{2}\geq u(\eta),
\end{equation*}
that is
\begin{equation*} \label{eq-18}
u(T)>u(\eta),
\end{equation*}
which is a contradiction to the fact that $u(t)$ is a monotone decreasing function.

If  $u(T)=0$, then $\int_{0}^{\eta }u(s)ds=0$, this is $u(t)\equiv 0$ for all $%
t\in [0,\eta ].$ If there exists  $t_{0}\in (\eta ,T)$ such that $%
u(t_{0})>0$, then  $\ u(0)=u(\eta )<u(t_{0})$, a contradiction with the fact that $u$ is monotone decreasing.
Therefore, no positive solutions exist.
\end{proof}

\begin{lemma}\label{lem 2.4}
Let $0<\alpha <\frac{1}{\eta}$. If $y\in
C([0,T],[0,\infty ))$, then the unique solution $u$ of
the problem \eqref{eq-9}-\eqref{eq-10} satisfies
\begin{equation} \label{eq-19}
\min_{t\in [0,T]}u(t)\geq \gamma \|u\|,
\end{equation}
where
\begin{equation} \label{eq-20}
\gamma=\frac{\alpha\eta(T-\eta)}{T-\alpha\eta^{2}}, \ \|u\|=\max_{t\in
[0,T]}|u(t)|.
\end{equation}
\end{lemma}
\begin{proof}
By Lemma \ref{lem 2.2}, we know that
\begin{equation*} \label{eq-21}
u(T)\leq u(t)\leq u(0).
\end{equation*}
So
\begin{equation} \label{eq-22}
\min_{t\in [0,T]}u(t)=u(T), \ \ \max_{t\in [0,T]}u(t)=u(0).
\end{equation}
From \eqref{eq-14} and the fact that $u$ is monotone decreasing, we get
\begin{equation} \label{eq-23}
u(T)=\alpha\int_{0}^{\eta }u(s)ds \geq \alpha\frac{\eta }{2}(u(0)+u(\eta ))\geq\alpha\eta{u(\eta)}.
\end{equation}
Using the concavity of $u$ and \eqref{eq-10},\eqref{eq-14} and \eqref{eq-23} we have
\begin{eqnarray*}
u(0)&\leq&u(T)+\frac{u(T)-u(\eta)}{T-\eta}(0-T)\\
&\leq&u(T)\left[1-T\frac{1-\frac{1}{\alpha\eta}}{T-\eta}\right]\\
&=&u(T)\frac{T-\alpha\eta^{2}}{\alpha\eta(T-\eta)}.
\end{eqnarray*}
Combining this with \eqref{eq-22}, we obtain
\begin{equation*} \label{eq-24}
\min_{t\in [0,T]}u(t)\geq \frac{\alpha\eta(T-\eta)}{T-\alpha\eta^{2}} \|u\|.
\end{equation*}
\end{proof}
Let $E = C([0,T],\mathbb{R})$, and only the sup norm is used. It is easy to see that the BVP
\eqref{eq-2},\eqref{eq-3} has a solution  $u = u(t)$ if and only if  $u$ is a fixed point of operator $A$, where
$A$ is defined by
\begin{equation}\label{eq-25}
\begin{split}
Au(t)&=\frac{1}{1-\alpha\eta}\int_{0}^{T}(T-s)a(s)f(u(s))ds \\
&-\frac{\alpha}{2(1-\alpha \eta)}\int_{0}^{\eta }(\eta -s)^{2}a(s)f(u(s))ds \\
&-\int_{0}^{t}(t-s)a(s)f(u(s))ds.
\end{split}
\end{equation}

Denote
\begin{equation}\label{eq-26}
K=\left\{u\in E: u\geq0,  \min_{t\in
[0,T]}u(t)\geq \gamma \|u\|\right\},
\end{equation}
where $\gamma$ is defined in \eqref{eq-20}.
It is obvious that $K$ is a cone in $E$. Moreover, by Lemma \ref{lem 2.2} and Lemma \ref{lem 2.4}, $AK\subset
K$. It is also easy to see that $A:K\rightarrow K$ is completely
continuous.\\

In what follows, for the sake of convenience, set
\begin{equation*}\label{eq-27}
\Lambda_{1}=\frac{1-\alpha\eta}{\int_{0}^{T}(T-s)a(s)ds},
\end{equation*}
\begin{equation*}\label{eq-28}
\Lambda_{2}=\frac{1-\alpha\eta}{\gamma\left(\int_{\eta}^{T}(T-s)a(s)ds+\frac{1}{2}\int_{0}^{\eta}\left[2(T-\eta)+\alpha(\eta^{2}-s^{2})\right]a(s)ds\right)}.
\end{equation*}

\section{The existence results of the BVP \eqref{eq-2},\eqref{eq-3} for the case: $f_{0}=0,\ f_{\infty}=\infty$
or  $f_{0}=\infty,\ f_{\infty}=0$}

In this section, we establish the existence of single positive solution for the
BVP \eqref{eq-2},\eqref{eq-3} under $f_{0}=0,\ f_{\infty}=\infty$ or  $f_{0}=\infty,\ f_{\infty}=0$

\begin{theorem}\label{theo 3.1}
The BVP \eqref{eq-2},\eqref{eq-3} has at least one positive solution in the case
\begin{itemize}
\item[(H1)]
$f_{0}=0,\ f_{\infty}=\infty$ or
\item[(H2)] $f_{0}=\infty,\ f_{\infty}=0$.
\end{itemize}
\end{theorem}

\begin{proof}
At first, let {\rm (H1)} hold. Since $f_{0}=0$, then for any $\epsilon\in\left(0,\Lambda_{1}\right]$, there
exists $\rho_{\star}>0$ such that
\begin{equation}\label{eq-29}
f(u)\leq \epsilon{u},\ \text{for}\ u\in(0,\rho_{\star}].
\end{equation}
Let $\Omega_{\rho_{\star}}=\left\{u\in E: \|u\|<\rho_{\star}\right\}$, then from \eqref{eq-25},\eqref{eq-29},
for any $u\in K\cap \partial\Omega_{\rho_{\star}}$, we have
\begin{eqnarray*}
Au(t)&=& \frac{1}{1-\alpha\eta}\int_{0}^{T}(T-s)a(s)f(u(s))ds \\
&&-\frac{\alpha}{2(1-\alpha \eta)}\int_{0}^{\eta }(\eta -s)^{2}a(s)f(u(s))ds \\
&&-\int_{0}^{t}(t-s)a(s)f(u(s))ds
\end{eqnarray*}
\begin{eqnarray*}
&\leq& \frac{1}{1-\alpha\eta}\int_{0}^{T}(T-s)a(s)f(u(s))ds \\
&\leq&\frac{\epsilon\rho_{\star}}{1-\alpha\eta}\int_{0}^{T}(T-s)a(s)ds \\
&=&\frac{\epsilon}{\Lambda_{1}}\rho_{\star} \\
&\leq&\rho_{\star}=\|u\|,
\end{eqnarray*}
which yields
\begin{equation}\label{eq-30}
\|Au\|\leq\|u\|, \ \ \text{for}\ u\in K\cap\partial\Omega_{\rho_{\star}}.
\end{equation}

Further, since $f_{\infty}=\infty$, then for any $M^{\star}\in\left[\Lambda_{2},\infty\right)$, there exists
$\rho^{\star}>\rho_{\star}$ such that
\begin{equation}\label{eq-31}
f(u)\geq M^{\star}u,\ \text{for}\ u\geq \gamma\rho^{\star}.
\end{equation}

Now, set $\Omega_{\rho^{\star}}=\left\{u\in E: \|u\|<\rho^{\star}\right\}$ for $u\in K\cap
\partial\Omega_{\rho^{\star}}$.

Since $u\in K$, $\min_{t\in\left[0,T\right]}u(t)\geq\gamma\|u\|=\gamma\rho^{\star}$. Hence, for any $u\in K\cap
\partial\Omega_{\rho^{\star}}$, from \eqref{eq-25},\eqref{eq-31}, we get
\begin{eqnarray*}
Au(\eta)&=&\frac{1}{1-\alpha\eta}\int_{0}^{T}(T-s)a(s)f(u(s))ds \\
&&-\frac{\alpha}{2(1-\alpha \eta)}\int_{0}^{\eta }(\eta -s)^{2}a(s)f(u(s))ds \\
&&-\int_{0}^{\eta}(\eta-s)a(s)f(u(s))ds\\
&=& \frac{T}{1-\alpha\eta}\int_{0}^{T}a(s)f(u(s))ds -\frac{1}{1-\alpha \eta}\int_{0}^{\eta }sa(s)f(u(s))ds \\
&&- \frac{1}{1-\alpha\eta}\int_{\eta}^{T}sa(s)f(u(s))ds -\frac{\alpha\eta^{2}}{2(1-\alpha \eta)}\int_{0}^{\eta
}a(s)f(u(s))ds  \\
&&+\frac{\alpha\eta}{1-\alpha \eta}\int_{0}^{\eta}sa(s)f(u(s))ds -
\frac{\alpha}{2(1-\alpha\eta)}\int_{0}^{\eta}s^{2}a(s)f(u(s))ds \\
&&-\eta\int_{0}^{\eta}a(s)f(u(s))ds +\int_{0}^{\eta}sa(s)f(u(s))ds\\
&=&\frac{T}{1-\alpha\eta}\int_{0}^{T}a(s)f(u(s))ds -\frac{1}{1-\alpha \eta}\int_{\eta}^{T}sa(s)f(u(s))ds \\
&&-\frac{\alpha\eta^{2}}{2(1-\alpha \eta)}\int_{0}^{\eta }a(s)f(u(s))ds -\frac{\alpha}{2(1-\alpha
\eta)}\int_{0}^{\eta }s^{2}a(s)f(u(s))ds \\
&&-\eta\int_{0}^{\eta }a(s)f(u(s))ds
\end{eqnarray*}
\begin{eqnarray*}
&=&\frac{1}{2(1-\alpha \eta)}\int_{0}^{\eta}[2(T-\eta)+\alpha(\eta^{2}-s^{2})]a(s)f(u(s))ds \\
&&+\frac{1}{1-\alpha\eta}\int_{\eta}^{T}(T-s)a(s)f(u(s))ds \\
&\geq&\frac{\gamma\rho^{\star}M^{\star}}{1-\alpha\eta}\left(\int_{\eta}^{T}(T-s)a(s)ds
+\frac{1}{2}\int_{0}^{\eta}[2(T-\eta)+\alpha(\eta^{2}-s^{2})]a(s)ds\right) \\
&=&\frac{M^{\star}}{\Lambda_{2}}\rho^{\star}\\
&\geq&\rho^{\star}=\|u\|,
\end{eqnarray*}
which implies
\begin{equation}\label{eq-32}
 \|Au\|\geq\|u\|,\ \ \text{for}\  u\in K\cap\partial\Omega_{\rho^{\star}}.
\end{equation}
Therefore, from \eqref{eq-30},\eqref{eq-32} and Theorem \ref{theo 2.1}, it follows that $A$ has a fixed point in
$K\cap (\overline{\Omega}_{\rho^{\star}}\backslash \Omega _{\rho_{\star}})$ such that
$\rho_{\star}\leq\|u\|\leq\rho^{\star}$.

Next, let {\rm (H2)} hold. In view of $f_{0}=\infty$, for any $M_{\star}\in\left[\Lambda_{2},\infty\right)$,
there exists $r_{\star}>0$ such that
\begin{equation}\label{eq-33}
f(u)\geq M_{\star}u,\ \text{for}\ u\in(0,r_{\star}].
\end{equation}

Set $\Omega_{r_{\star}}=\left\{u\in E: \|u\|<r_{\star}\right\}$,
for $u\in K\cap\partial\Omega_{r_{\star}}$. Since $u\in K$, it follows that
$\min_{t\in\left[0,T\right]}u(t)\geq\gamma\|u\|=\gamma r_{\star}$.
Thus, from \eqref{eq-25}, \eqref{eq-33}, for any $u\in
K\cap\partial\Omega_{r_{\star}}$, we can get

\begin{eqnarray*}
Au(\eta)&=&\frac{1}{1-\alpha\eta}\int_{0}^{T}(T-s)a(s)f(u(s))ds \\
&&-\frac{\alpha}{2(1-\alpha \eta)}\int_{0}^{\eta }(\eta -s)^{2}a(s)f(u(s))ds \\
&&-\int_{0}^{\eta}(\eta-s)a(s)f(u(s))ds \\
&=&\frac{1}{2(1-\alpha \eta)}\int_{0}^{\eta}[2(T-\eta)+\alpha(\eta^{2}-s^{2})]a(s)f(u(s))ds \\
&&+\frac{1}{1-\alpha\eta}\int_{\eta}^{T}(T-s)a(s)f(u(s))ds \\
&\geq&\frac{\gamma r_{\star}M_{\star}}{1-\alpha\eta}\left(\int_{\eta}^{T}(T-s)a(s)ds
+\frac{1}{2}\int_{0}^{\eta}[2(T-\eta)+\alpha(\eta^{2}-s^{2})]a(s)ds\right) \\
&=&\frac{M_{\star}}{\Lambda_{2}}r_{\star}\\
&\geq& r_{\star}=\|u\|,
\end{eqnarray*}
which implies

\begin{equation}\label{eq-34}
\|Au\|\geq \|u\|,\ \ \text{for}\ u\in K\cap\partial\Omega_{r_{\star}}.
\end{equation}
Again, since $f_{\infty}=0$, then for any $\epsilon_{1}\in\left(0,\Lambda_{1}\right)$, there exists
$r_{0}>r_{\star}$ such that
\begin{equation}\label{eq-35}
f(u)\leq\epsilon_{1}u,\ \text{for}\ u\in\left[r_{0},\infty\right).
\end{equation}

We consider two cases:

Case (i). Suppose that $f(u)$ is unbounded, then from $f\in C([0,\infty),[0,\infty))$, we know that there is
$r^{\star}>r_{0}$ such  that
\begin{equation}\label{eq-36}
f(u)\leq f(r^{\star}),\ \text{for}\ u\in\left[0,r^{\star}\right].
\end{equation}
Since $r^{\star}>r_{0}$, then from \eqref{eq-35} and \eqref{eq-36}, one has
\begin{equation}\label{eq-37}
f(u)\leq f(r^{\star})\leq\epsilon_{1}r^{\star},\ \text{for}\ u\in\left[0,r^{\star}\right].
\end{equation}
For $u\in K$ and $\|u\|=r^{\star}$, from \eqref{eq-25} and \eqref{eq-37}, we obtain

\begin{eqnarray*}
Au(t)&\leq& \frac{1}{1-\alpha\eta}\int_{0}^{T}(T-s)a(s)f(u(s))ds \\
&\leq&\frac{\epsilon_{1}r^{\star}}{1-\alpha\eta}\int_{0}^{T}(T-s)a(s)ds \\
&=&\frac{\epsilon_{1}}{\Lambda_{1}}r^{\star} \\
&\leq&r^{\star}=\|u\|.
\end{eqnarray*}

Case (ii). Suppose that $f(u)$ is bounded, say $f(u)\leq L$ for all $u\in [0,\infty)$. Taking
$r^{\star}\geq\max\left\{\frac{L}{\epsilon_{1}},r_{\star}\right\}$.

For $u\in K$ with $\|u\|=r^{\star}$, from \eqref{eq-25}, one has
\begin{eqnarray*}
Au(t)&\leq& \frac{1}{1-\alpha\eta}\int_{0}^{T}(T-s)a(s)f(u(s))ds \\
&\leq&\frac{L}{1-\alpha\eta}\int_{0}^{T}(T-s)a(s)ds \\
&\leq&\frac{\epsilon_{1}r^{\star}}{1-\alpha\eta}\int_{0}^{T}(T-s)a(s)ds \\
&=&\frac{\epsilon_{1}}{\Lambda_{1}}r^{\star} \\
&\leq&r^{\star}=\|u\|.
\end{eqnarray*}
Hence, in either case, we always may set $\Omega_{r^{\star}}=\left\{u\in E: \|u\|<r^{\star}\right\}$ such that
\begin{equation}\label{eq-38}
\|Au\|\leq \|u\|,\ \ \text{for}\ u\in K\cap\partial\Omega_{r^{\star}}.
\end{equation}
Hence, from \eqref{eq-34}, \eqref{eq-38} and Theorem \ref{theo 2.1}, it follows that $A$ has a fixed point $u$
in
$K\cap (\overline{\Omega}_{r^{\star}}\backslash \Omega _{r_{\star}})$ such that $r_{\star}\leq\|u\|\leq
r^{\star}$.

The proof is therefore complete.
\end{proof}

\section{The existence results of the BVP \eqref{eq-2},\eqref{eq-3} for the case: $f_{0}=f_{\infty}=\infty$ or
$f_{0}=f_{\infty}=0$}

\begin{theorem}\label{theo 4.1}
Assume that the following assumptions are satisfied.
\begin{itemize}
\item[(H3)]
$f_{0}=f_{\infty}=\infty$.
\item[(H4)] There exist constants $\rho_{1}> 0$ and $M_{1}\in(0,\Lambda_{1}]$ such that
$f(u)\leq M_{1}\rho_{1}$, for $u\in[0,\rho_{1}]$ .
\end{itemize}
Then, the BVP \eqref{eq-2},\eqref{eq-3} has at least two positive solutions $u_{1}$ and $u_{2}$
such that
\begin{equation*}\label{eq-39}
0<\|u_{1}\|<\rho_{1}<\|u_{2}\|.
\end{equation*}
\end{theorem}
\begin{proof}
At first, in view of $f_{0}=\infty$, then for any $M_{\star}\in\left[\Lambda_{2},\infty\right)$, there exists
$\rho_{\star}\in\left(0,\rho_{1}\right)$ such that
\begin{equation}\label{eq-40}
f(u)\geq M_{\star}u,\ \  0<u\leq \rho_{\star}.
\end{equation}

Set $\Omega_{\rho_{\star}}=\left\{u\in E: \|u\|<\rho_{\star}\right\}$ for $u\in K\cap
\partial\Omega_{\rho_{\star}}$.

Since $u\in K$, $\min_{t\in\left[0,T\right]}u(t)\geq\gamma\|u\|=\gamma\rho_{\star}$. Hence, for any $u\in K\cap
\partial\Omega_{\rho_{\star}}$, from \eqref{eq-25},\eqref{eq-40}, we have

\begin{eqnarray*}
Au(\eta)&=&\frac{1}{1-\alpha\eta}\int_{0}^{T}(T-s)a(s)f(u(s))ds \\
&&-\frac{\alpha}{2(1-\alpha \eta)}\int_{0}^{\eta }(\eta -s)^{2}a(s)f(u(s))ds \\
&&-\int_{0}^{\eta}(\eta-s)a(s)f(u(s))ds \\
&\geq&\frac{\gamma \rho_{\star}M_{\star}}{1-\alpha\eta}\left(\int_{\eta}^{T}(T-s)a(s)ds
+\frac{1}{2}\int_{0}^{\eta}[2(T-\eta)+\alpha(\eta^{2}-s^{2})]a(s)ds\right) \\
&=&\rho_{\star}M_{\star}\Lambda_{2}^{-1}\\
&\geq&\rho_{\star}=\|u\|,
\end{eqnarray*}
which implies

\begin{equation}\label{eq-41}
\|Au\|\geq \|u\|,\ \ \text{for}\ u\in K\cap\partial\Omega_{\rho_{\star}}.
\end{equation}

Next, since $f_{\infty}=\infty$, then for any $M^{\star}\in\left[\Lambda_{2},\infty\right)$, there exists
$\rho^{\star}>\rho_{1}$ such that
\begin{equation}\label{eq-42}
f(u)\geq M^{\star}u,\ \ \text{for}\ u\geq \gamma\rho^{\star}.
\end{equation}

Set $\Omega_{\rho^{\star}}=\left\{u\in E: \|u\|<\rho^{\star}\right\}$ for $u\in K\cap
\partial\Omega_{\rho^{\star}}$.

Since $u\in K$, $\min_{t\in\left[0,T\right]}u(t)\geq\gamma\|u\|=\gamma\rho^{\star}$. Hence, for any $u\in K\cap
\partial\Omega_{\rho^{\star}}$, from \eqref{eq-25},\eqref{eq-42}, we can get

\begin{eqnarray*}
Au(\eta)&=&\frac{1}{1-\alpha\eta}\int_{0}^{T}(T-s)a(s)f(u(s))ds \\
&&-\frac{\alpha}{2(1-\alpha \eta)}\int_{0}^{\eta }(\eta -s)^{2}a(s)f(u(s))ds \\
&&-\int_{0}^{\eta}(\eta-s)a(s)f(u(s))ds \\
&\geq&\frac{\gamma \rho^{\star}M^{\star}}{1-\alpha\eta}\left(\int_{\eta}^{T}(T-s)a(s)ds
+\frac{1}{2}\int_{0}^{\eta}[2(T-\eta)+\alpha(\eta^{2}-s^{2})]a(s)ds\right) \\
&=&\rho^{\star}M^{\star}\Lambda_{2}^{-1}\\
&\geq&\rho^{\star}=\|u\|,
\end{eqnarray*}
which implies
\begin{equation}\label{eq-43}
 \|Au\|\geq\|u\|,\ \ \text{for}\  u\in K\cap\partial\Omega_{\rho^{\star}}.
\end{equation}

Finally, let $\Omega_{\rho_{1}}=\left\{u\in E: \|u\|<\rho_{1}\right\}$. By {\rm (H4)}, for any $u\in
K\cap\partial\Omega_{\rho_{1}}$, we have
\begin{eqnarray*}
Au(t)&\leq& \frac{1}{1-\alpha\eta}\int_{0}^{T}(T-s)a(s)f(u(s))ds \\
&\leq&\frac{M_{1}\rho_{1}}{1-\alpha\eta}\int_{0}^{T}(T-s)a(s)ds \\
&=&\rho_{1}M_{1}\Lambda_{1}^{-1}\\
&\leq&\rho_{1}=\|u\|.
\end{eqnarray*}

Which yields
\begin{equation}\label{eq-44}
\|Au\|\leq\|u\|,\ \ \text{for}\ u\in K\cap\partial\Omega_{\rho_{1}}.
\end{equation}

Hence, since $\rho_{\star}<\rho_{1}<\rho^{\star}$ and from \eqref{eq-41}, \eqref{eq-43}, \eqref{eq-44}, it
follows from Theorem \ref{theo 2.1} that $A$ has a fixed point $u_{1}$ in
$K\cap (\overline{\Omega}_{\rho_{1}}\backslash \Omega _{\rho_{\star}})$ and a fixed point $u_{2}$ in $K\cap
(\overline{\Omega}_{\rho^{\star}}\backslash \Omega _{\rho_{1}})$. Both are positive solutions of the BVP
\eqref{eq-2},\eqref{eq-3} and $0<\|u_{1}\|<\rho_{1}<\|u_{2}\|$.

The proof is therefore complete.
\end{proof}

\begin{theorem}\label{theo 4.2}
Assume that the following assumptions are satisfied.
\begin{itemize}
\item[(H5)]
 $f_{0}=f_{\infty}=0$.
\item[(H6)] There exist constants $\rho_{2}> 0$ and $M_{2}\in\left[\Lambda_{2},\infty\right)$  such that
$f(u)\geq M_{2}\rho_{2}$, for $u\in[\gamma\rho_{2} ,\rho_{2}]$ .
\end{itemize}
Then, the BVP \eqref{eq-2},\eqref{eq-3} has at least two positive solutions $u_{1}$ and $u_{2}$
such that
\begin{equation*}\label{eq-45}
0<\|u_{1}\|<\rho_{2}<\|u_{2}\|.
\end{equation*}
\end{theorem}

\begin{proof}
Firstly, since $f_{0}=0$, for any $\epsilon\in\left(0,\Lambda_{1}\right]$, there exists
$\rho_{\star}\in\left(0,\rho_{2}\right)$ such that
\begin{equation}\label{eq-46}
f(u)\leq\epsilon u ,\ \text{for} \ u\in(0,\rho_{\star}].
\end{equation}
Let $\Omega_{\rho_{\star}}=\left\{u\in E: \|u\|<\rho_{\star}\right\}$ for any $u\in K\cap
\partial\Omega_{\rho_{\star}}$.
Then, from \eqref{eq-25},\eqref{eq-46}, we obtain
\begin{eqnarray*}
Au(t)&\leq& \frac{1}{1-\alpha\eta}\int_{0}^{T}(T-s)a(s)f(u(s))ds \\
&\leq&\frac{\epsilon\rho_{\star}}{1-\alpha\eta}\int_{0}^{T}(T-s)a(s)ds \\
&=&\epsilon\Lambda_{1}^{-1}\rho_{\star}\\
&\leq&\rho_{\star}=\|u\|.
\end{eqnarray*}

Which yields
\begin{equation}\label{eq-47}
\|Au\|\leq\|u\|,\ \ \text{for}\ u\in K\cap\partial\Omega_{\rho_{\star}}.
\end{equation}
Secondly, in view of $f_{\infty}=0$, for any $\epsilon_{1}\in\left(0,\Lambda_{1}\right]$, there exists
$\rho_{0}>\rho_{2}$ such that
\begin{equation}\label{eq-48}
f(u)\leq\epsilon_{1} u , \ \ \text{for}\  u\in\left[\rho_{0},\infty\right).
\end{equation}

We consider two cases:

Case (i). Suppose that $f(u)$ is unbounded. Then from $f\in C([0,\infty),[0,\infty))$, we know that there
is $\rho^{\star}>\rho_{0}$ such that
\begin{equation}\label{eq-49}
f(u)\leq f(\rho^{\star}),   \ \ \text{for}\  u\in\left[0,\rho^{\star}\right].
\end{equation}
Since $\rho^{\star}>\rho_{0}$, then from \eqref{eq-48}, \eqref{eq-49}, one has
\begin{equation}\label{eq-50}
f(u)\leq f(\rho^{\star})\leq\epsilon_{1}\rho^{\star},   \ \ \text{for} \  u\in\left[0,\rho^{\star}\right].
\end{equation}
For $u\in K$ and $ \|u\|=\rho^{\star}$ , from \eqref{eq-25}, \eqref{eq-50}, we obtain
\begin{eqnarray*}
Au(t)&\leq& \frac{1}{1-\alpha\eta}\int_{0}^{T}(T-s)a(s)f(u(s))ds \\
&\leq&\frac{\epsilon_{1}\rho^{\star}}{1-\alpha\eta}\int_{0}^{T}(T-s)a(s)ds \\
&=&\epsilon_{1}\Lambda_{1}^{-1}\rho^{\star}\\
&\leq&\rho^{\star}=\|u\|.
\end{eqnarray*}

Case (ii). Suppose that $f(u)$ is bounded, say $f(u)\leq L$ for all $u\in \left[0,\infty\right)$. Taking
$\rho^{\star}\geq \max\left\{\frac{L}{\epsilon_{1}}, \rho_{0}\right\}$. For $u\in K$ with $ \|u\|=\rho^{\star}$,
from \eqref{eq-25}, we have
\begin{eqnarray*}
Au(t)&\leq& \frac{1}{1-\alpha\eta}\int_{0}^{T}(T-s)a(s)f(u(s))ds \\
&\leq&\frac{L}{1-\alpha\eta}\int_{0}^{T}(T-s)a(s)ds \\
&\leq&\frac{\epsilon_{1}\rho^{\star}}{1-\alpha\eta}\int_{0}^{T}(T-s)a(s)ds \\
&=&\epsilon_{1}\Lambda_{1}^{-1}\rho^{\star}\\
&\leq&\rho^{\star}=\|u\|.
\end{eqnarray*}
Hence, in either case, we always may set $\Omega_{\rho^{\star}}=\left\{u\in E: \|u\|<\rho^{\star}\right\}$ such
that
\begin{equation}\label{eq-51}
\|Au\|\leq\|u\|,\ \ \text{for}\ u\in K\cap\partial\Omega_{\rho^{\star}}.
\end{equation}
Finally, set $\Omega_{\rho_{2}}=\left\{u\in E: \|u\|<\rho_{2}\right\}$ for $u\in K\cap
\partial\Omega_{\rho_{2}}$.
Since $u\in K$, $\min_{t\in [0,T]}u(t)\geq \gamma \|u\|=\gamma \rho_{2}$. Hence, for any $u\in
K\cap\partial\Omega_{\rho_{2}}$, from \eqref{eq-25} and {\rm (H6)}, we can get
\begin{eqnarray*}
Au(\eta)&=&\frac{1}{1-\alpha\eta}\int_{0}^{T}(T-s)a(s)f(u(s))ds \\
&&-\frac{\alpha}{2(1-\alpha \eta)}\int_{0}^{\eta }(\eta -s)^{2}a(s)f(u(s))ds \\
&&-\int_{0}^{\eta}(\eta-s)a(s)f(u(s))ds \\
&\geq&\frac{\gamma \rho_{2}M_{2}}{1-\alpha\eta}\left(\int_{\eta}^{T}(T-s)a(s)ds
+\frac{1}{2}\int_{0}^{\eta}[2(T-\eta)+\alpha(\eta^{2}-s^{2})]a(s)ds\right) \\
&=&\rho_{2}M_{2}\Lambda_{2}^{-1}\\
&\geq&\rho_{2}=\|u\|,
\end{eqnarray*}
which implies
\begin{equation}\label{eq-53}
\|Au\|\geq\|u\|,\ \ \text{for}\ u\in K\cap\partial\Omega_{\rho_{2}}.
\end{equation}
Hence, since $\rho_{\star}<\rho_{2}<\rho^{\star}$ and from \eqref{eq-47}, \eqref{eq-51} and \eqref{eq-53}, it
follows from
Theorem \ref{theo 2.1} that $A$ has a fixed point $u_{1}$ in $K\cap (\overline{\Omega}_{\rho_{2}}\backslash
\Omega _{\rho_{\star}})$ and a fixed point $u_{2}$ in $K\cap (\overline{\Omega}_{\rho^{\star}}\backslash \Omega
_{\rho_{2}})$. Both are positive solutions of the BVP \eqref{eq-2},\eqref{eq-3} and
$0<\|u_{1}\|<\rho_{2}<\|u_{2}\|$.
The proof is therefore complete.
\end{proof}

\section{The existence results of the BVP \eqref{eq-2},\eqref{eq-3} for the case: $f_{0}, f_{\infty}\not\in
\left\{0,\infty\right\}$}
In this section, we discuss the existence for the positive solution of the BVP \eqref{eq-2},\eqref{eq-3}
assuming $f_{0}, f_{\infty}\not\in \left\{0,\infty\right\}$.

Now, we shall state and prove the following main result.
\begin{theorem}\label{theo 5.1}
Suppose {\rm (H4)} and {\rm (H6)} hold and that $\rho_{1}\neq \rho_{2}$. Then, the BVP \eqref{eq-2},\eqref{eq-3}
has
at least one positive solution $u$ satisfying $\rho_{1}<\|u\|<\rho_{2}$ or $\rho_{2}<\|u\|<\rho_{1}$.
\end{theorem}
\begin{proof}
Without loss of generality, we may assume that $\rho_{1}<\rho_{2}$.

Let $\Omega_{\rho_{1}}=\left\{u\in E: \|u\|<\rho_{1}\right\}$. By {\rm (H4)}, for any $u\in
K\cap\partial\Omega_{\rho_{1}}$, we obtain
\begin{eqnarray*}
Au(t)&\leq& \frac{1}{1-\alpha\eta}\int_{0}^{T}(T-s)a(s)f(u(s))ds \\
&\leq&\frac{M_{1}\rho_{1}}{1-\alpha\eta}\int_{0}^{T}(T-s)a(s)ds \\
&=&\rho_{1}M_{1}\Lambda_{1}^{-1}\\
&\leq&\rho_{1}=\|u\|.
\end{eqnarray*}
which yields
\begin{equation}\label{eq-54}
\|Au\|\leq\|u\| , \ u\in K\cap\partial\Omega_{\rho_{1}}.
\end{equation}

Now, set $\Omega_{\rho_{2}}=\left\{u\in E: \|u\|<\rho_{2}\right\}$, for  $u\in K\cap\partial\Omega_{\rho_{2}}$.
Since $u\in K$, $\min_{t\in
[\eta,T]}u(t)\geq \gamma \|u\|=\gamma \rho_{2}$.

Hence, for any $u\in K\cap\partial\Omega_{\rho_{2}}$, from \eqref{eq-25} and  {\rm (H6)}, we can get

\begin{eqnarray*}
Au(\eta)&=&\frac{1}{2(1-\alpha \eta)}\int_{0}^{\eta}[2(T-\eta)+\alpha(\eta^{2}-s^{2})]a(s)f(u(s))ds \\
&&+\frac{1}{1-\alpha\eta}\int_{\eta}^{T}(T-s)a(s)f(u(s))ds \\
&\geq&\frac{\gamma \rho_{2}M_{2}}{1-\alpha\eta}\left(\int_{\eta}^{T}(T-s)a(s)ds
+\frac{1}{2}\int_{0}^{\eta}[2(T-\eta)+\alpha(\eta^{2}-s^{2})]a(s)ds\right) \\
&=&\rho_{2}M_{2}\Lambda_{2}^{-1}\\
&\geq&\rho_{2}=\|u\|,
\end{eqnarray*}
which implies
\begin{equation}\label{eq-55}
\|Au\|\geq\|u\|,\  \text{for} \ u\in K\cap\partial\Omega_{\rho_{2}}.
\end{equation}
Hence, since $\rho_{1}<\rho_{2}$ and from \eqref{eq-54} and \eqref{eq-55}, it follows from Theorem \ref{theo
2.1} that
$A$ has a fixed point $u$ in $K\cap (\overline{\Omega}_{\rho_{2}}\backslash \Omega _{\rho_{1}})$. Moreover, it
is a positive solution of the BVP \eqref{eq-2},\eqref{eq-3} and
\begin{equation*}\label{eq-56}
\rho_{1}<\|u\|<\rho_{2}.
\end{equation*}
The proof is therefore complete.
\end{proof}

\begin{corollary}\label{cor 5.2}
Assume that the following assumptions hold.
\begin{itemize}
\item[(H7)] $f_{0}=\alpha_{1}\in \left[0,\theta_{1}\Lambda_{1}\right)$, where $\theta_{1}\in(0,1]$.
\item[(H8)] $f_{\infty}=\beta_{1}\in \left(\frac{\theta_{2}}{\gamma}\Lambda_{2},\infty\right)$, where
    $\theta_{2}\geq1$.
\end{itemize}
Then, the BVP \eqref{eq-2},\eqref{eq-3} has at least one positive solution.
\end{corollary}
\begin{proof}
In view of $f_{0}=\alpha_{1}\in \left[0,\theta_{1}\Lambda_{1}\right)$, for
$\epsilon=\theta_{1}\Lambda_{1}-\alpha_{1}>0$, there exists a sufficiently large $\rho_{1}>0$ such that
\begin{equation*}\label{eq-57}
f(u)\leq(\alpha_{1}+\epsilon)u=\theta_{1}\Lambda_{1} u\leq\theta_{1}\Lambda_{1}\rho_{1},\ \text{for}\
u\in(0,\rho_{1}].
\end{equation*}
Since $\theta_{1}\in(0,1]$, then $\theta_{1}\Lambda_{1}\in(0,\Lambda_{1}]$. By the inequality above, {\rm (H4)}
is satisfied.
Since $f_{\infty}=\beta_{1}\in \left(\frac{\theta_{2}}{\gamma}\Lambda_{2},\infty\right)$, for
$\epsilon=\beta_{1}-\frac{\theta_{2}}{\gamma}\Lambda_{2}>0$, there exists a sufficiently large
$\rho_{2}(>\rho_{1})$ such that
\begin{equation*}\label{eq-58}
\frac{f(u)}{u}\geq\beta_{1}-\epsilon=\frac{\theta_{2}}{\gamma}\Lambda_{2},\ \text{for}\
u\in\left[\gamma\rho_{2},\infty\right),
\end{equation*}
thus, when $u\in\left[\gamma\rho_{2},\rho_{2}\right]$, one has
\begin{equation*}\label{eq-59}
f(u)\geq\frac{\theta_{2}}{\gamma}\Lambda_{2}u\geq\theta_{2}\Lambda_{2}\rho_{2}.
\end{equation*}
Since $\theta_{2}\geq1$, $\theta_{2}\Lambda_{2}\in\left[\Lambda_{2},\infty\right)$, then from the above
inequality, condition {\rm (H6)} is satisfied.
Hence, from Theorem \ref{theo 5.1} , the desired result holds.
\end{proof}

\begin{corollary}\label{cor 5.3}
Assume that the following assumptions hold.
\begin{itemize}
\item[(H9)] $f_{0}=\alpha_{2}\in \left(\frac{\theta_{2}}{\gamma}\Lambda_{2},\infty\right)$, where
    $\theta_{2}\geq1$.
\item[(H10)] $f_{\infty}=\beta_{2}\in \left[0,\theta_{1}\Lambda_{1}\right)$, where $\theta_{1}\in(0,1]$.
\end{itemize}
Then, the BVP \eqref{eq-2},\eqref{eq-3} has at least one positive solution.
\end{corollary}
\begin{proof}
Since $f_{0}=\alpha_{2}\in \left(\frac{\theta_{2}}{\gamma}\Lambda_{2},\infty\right)$, for
$\epsilon=\alpha_{2}-\frac{\theta_{2}}{\gamma}\Lambda_{2}>0$, there exists a sufficiently
small $\rho_{2}>0$ such that
\begin{equation*}\label{eq-60}
\frac{f(u)}{u}\geq\alpha_{2}-\epsilon=\frac{\theta_{2}}{\gamma}\Lambda_{2},\ \text{for}\
u\in\left(0,\rho_{2}\right).
\end{equation*}
Thus, when $u\in\left[\gamma\rho_{2},\rho_{2}\right]$, one has
\begin{equation*}\label{eq-61}
f(u)\geq\frac{\theta_{2}}{\gamma}\Lambda_{2}u\geq\theta_{2}\Lambda_{2}\rho_{2}.
\end{equation*}
which yields the condition {\rm (H6)} of Theorem \ref{theo 4.2}.

In view of $f_{\infty}=\beta_{2}\in \left[0,\theta_{1}\Lambda_{1}\right)$, for
$\epsilon=\theta_{1}\Lambda_{1}-\beta_{2}>0$, there exists a sufficiently large $\rho_{0}(>\rho_{2})$ such that
\begin{equation}\label{eq-60}
\frac{f(u)}{u}\leq\beta_{2}+\epsilon=\theta_{1}\Lambda_{1},\ \text{for}\ u\in\left[\rho_{0},\infty\right).
\end{equation}
We consider the following two cases:

Case (i). Suppose that $f(u)$ is unbounded. Then from $f\in C([0,\infty),[0,\infty))$, we know that there
is $\rho_{1}>\rho_{0}$ such that
\begin{equation}\label{eq-61}
f(u)\leq f(\rho_{1}),   \ \ \text{for}\  u\in\left[0,\rho_{1}\right].
\end{equation}
Since $\rho_{1}>\rho_{0}$, then from \eqref{eq-60}, \eqref{eq-61}, one has
\begin{equation*}\label{eq-62}
f(u)\leq f(\rho_{1})\leq \theta_{1}\Lambda_{1}\rho_{1},   \ \ \text{for} \  u\in\left[0,\rho_{1}\right].
\end{equation*}
Since $\theta_{1}\in(0,1]$, then $\theta_{1}\Lambda_{1}\in(0,\Lambda_{1}]$. By the inequality above, {\rm (H4)}
is satisfied.

Case (ii). Suppose that $f(u)$ is bounded, say
\begin{equation}\label{eq-63}
f(u)\leq L, \ \  \text{for all} \ \ u\in \left[0,\infty\right)
\end{equation}

In this case, taking sufficiently large $\rho_{1}>\frac{L}{\theta_{1}\Lambda_{1}}$, then from \eqref{eq-63}, we
know
\begin{equation*}\label{eq-64}
f(u)\leq L\leq\theta_{1}\Lambda_{1}\rho_{1}, \ \ \text{for} \  u\in\left[0,\rho_{1}\right].
\end{equation*}
Since $\theta_{1}\in(0,1]$, then $\theta_{1}\Lambda_{1}\in(0,\Lambda_{1}]$. By the inequality above, {\rm (H4)}
is satisfied.
Hence, from Theorem \ref{theo 5.1}, we get the conclusion of Corollary \ref{cor 5.3}.
\end{proof}

\begin{corollary}\label{cor 5.4}
Assume that the previous hypotheses {\rm (H4)}, {\rm (H8)} and {\rm (H9)} hold. Then, the BVP
\eqref{eq-2},\eqref{eq-3} has at least two positive solutions $u_{1}$ and $u_{2}$ such that
\begin{equation*}\label{eq-65}
0<\|u_{1}\|<\rho_{1}<\|u_{2}\|.
\end{equation*}
\end{corollary}
\begin{proof}
From {\rm (H8)} and the proof of Corollary \ref{cor 5.2}, we know that there exists a sufficiently large
$\rho_{2}>\rho_{1}$,  such that
\begin{equation*}\label{eq-66}
f(u)\geq\theta_{2}\Lambda_{2}\rho_{2}=M_{2}\rho_{2},   \ \ \text{for}\
u\in\left[\gamma\rho_{2},\rho_{2}\right],
\end{equation*}
where $M_{2}=\theta_{2}\Lambda_{2}\in \left[\Lambda_{2},\infty\right)$.

In view of {\rm (H9)} and the proof of Corollary \ref{cor 5.3}, we see that there exists a sufficiently small
$\rho_{2}^{\star}\in\left(0,\rho_{1}\right)$ such that
\begin{equation*}\label{eq-67}
f(u)\geq\theta_{2}\Lambda_{2}\rho_{2}^{\star}=M_{2}\rho_{2}^{\star},   \ \ \text{for}\
u\in\left[\gamma\rho_{2}^{\star},\rho_{2}^{\star}\right],
\end{equation*}
where $M_{2}=\theta_{2}\Lambda_{2}\in \left[\Lambda_{2},\infty\right)$.

Using this and {\rm (H4)}, we know by Theorem \ref{theo 5.1} that the BVP \eqref{eq-2},\eqref{eq-3} has two
positive
solutions $u_{1}$ and $u_{2}$ such that
\begin{equation*}\label{eq-68}
\rho_{2}^{\star}<\|u_{1}\|<\rho_{1}<\|u_{2}\|<\rho_{2}.
\end{equation*}
Thus, the proof is complete.
\end{proof}

\begin{corollary}\label{cor 5.5}
Assume that the previous hypotheses {\rm (H6)}, {\rm (H7)} and {\rm (H10)} hold. Then, the BVP
\eqref{eq-2},\eqref{eq-3} has at least two positive solutions $u_{1}$ and $u_{2}$ such that
\begin{equation*}\label{eq-69}
0<\|u_{1}\|<\rho_{2}<\|u_{2}\|.
\end{equation*}
\end{corollary}
\begin{proof}
By {\rm (H7)} and the proof of Corollary \ref{cor 5.2}, we obtain that there exists sufficiently small
$\rho_{1}\in\left(0,\rho_{2}\right)$ such that
\begin{equation*}\label{eq-70}
f(u)\leq\theta_{1}\Lambda_{1}\rho_{1}=M_{1}\rho_{1}, \ \ \text{for}\  u\in(0,\rho_{1}],
\end{equation*}
where $M_{1}=\theta_{1}\Lambda_{1}\in\left(0,\Lambda_{1}\right]$.

In view of {\rm (H10)} and the proof of Corollary \ref{cor 5.3}, there exists a sufficiently large
$\rho^{\star}_{1}>\rho_{2}$ such
that
\begin{equation*}\label{eq-71}
f(u)\leq\theta_{1}\Lambda_{1}\rho^{\star}_{1}=M_{1}\rho^{\star}_{1}, \ \ \text{for}\  u\in(0,\rho^{\star}_{1}],
\end{equation*}
where $M_{1}=\theta_{1}\Lambda_{1}\in\left(0,\Lambda_{1}\right]$.

Using this and {\rm (H6)}, we see by Theorem \ref{theo 5.1} that the BVP \eqref{eq-2},\eqref{eq-3} has two
positive solutions
$u_{1}$ and $u_{2}$ such that
\begin{equation*}\label{eq-72}
\rho_{1}<\|u_{1}\|<\rho_{2}<\|u_{2}\|<\rho^{\star}_{1}.
\end{equation*}
This completes the proof.
\end{proof}

\section{Examples}
In this section we present some examples to illustrate our main results.

\begin{exmp}
Consider the boundary value problem

\begin{equation}\label{eq-73}
{u^{\prime \prime }}(t)+tu^{p}=0, \  \ 0<t<1,
\end{equation}

\begin{equation}\label{eq-74}
{u^{\prime}}(0)=0, \  \ u(1)= 2\int_{0}^{\frac{1}{4}}u(s)ds.
\end{equation}

Set $\alpha=2$, $\eta=1/4$, $T=1$, $a(t)=t$, $f(u)=u^{p}$ $(p\in (0, 1)\cup(1, \infty)$. We can show that
$0<\alpha=2<4=1/\eta$.\\
Now we consider the existence of positive solutions of the problem \eqref{eq-73}, \eqref{eq-74} in two
cases.\\
Case 1: $p>1$. In this case, $f_{0}=0$, $f_{\infty}=\infty$ and {\rm (H1)} holds. Then, by Theorem \ref{theo
3.1}, the BVP \eqref{eq-73}, \eqref{eq-74} has at least one positive solution.\\
Case 2: $p\in (0, 1)$. In this case, $f_{0}=\infty$, $f_{\infty}=0$ and {\rm (H2)} holds. Then, by Theorem
\ref{theo 3.1}, the BVP \eqref{eq-73}, \eqref{eq-74} has at least one positive solution.
\end{exmp}

\begin{exmp}
Consider the boundary value problem

\begin{equation}\label{eq-75}
{u^{\prime \prime }}(t)+t^{2}u^{2}e^{u}=0, \  \ 0<t<\frac{3}{4},
\end{equation}

\begin{equation}\label{eq-76}
{u^{\prime}}(0)=0, \  \ u(\frac{3}{4})= \frac{3}{2}\int_{0}^{\frac{1}{2}}u(s)ds.
\end{equation}

Set $\alpha=3/2$, $\eta=1/2$, $T=3/4$, $a(t)=t^{2}$, $f(u)=u^{2}e^{u}$. We can show that
$0<\alpha=3/2<2=1/\eta$. Through a simple calculation we can get $f_{0}=0$ and $f_{\infty}=\infty$, which
implies that the condition {\rm (H1)} holds. Hence, by Theorem \ref{theo 3.1}, BVP \eqref{eq-75}, \eqref{eq-76})
has at least one positive solution.
\end{exmp}

\begin{exmp}
Consider the boundary value problem

\begin{equation}\label{eq-77}
{u^{\prime \prime }}(t)+e^{t}\frac{\sin{u}}{u^{2}}=0, \  \ 0<t<1,
\end{equation}

\begin{equation}\label{eq-78}
{u^{\prime}}(0)=0, \  \ u(1)= \frac{1}{2}\int_{0}^{\frac{1}{3}}u(s)ds.
\end{equation}

Set $\alpha=1/2$, $\eta=1/3$, $T=1$, $a(t)=e^{t}$, $f(u)=\sin{u}/{u^{2}}$. We can show that
$0<\alpha=1/2<3=1/\eta$. Through a simple calculation we can get $f_{0}=\infty$ and $f_{\infty}=0$, which
implies that the condition {\rm (H2)} holds. Hence, by Theorem \ref{theo 3.1}, BVP \eqref{eq-77}, \eqref{eq-78})
has at least one positive solution.
\end{exmp}

\begin{exmp}
\begin{equation}\label{eq-79}
{u^{\prime \prime }}(t)+\frac{5}{32}(2-t)^{3}(\frac{u^{\frac{1}{2}}}{2}+\frac{u^{2}}{32})=0, \  \ 0<t<2,
\end{equation}

\begin{equation}\label{eq-80}
{u^{\prime}}(0)=0, \  \ u(2)=2 \int_{0}^{1}u(s)ds.
\end{equation}
Set $\alpha=1/2$, $\eta=1$, $T=2$, $a(t)=\frac{5}{32}(2-t)^{3}$,
$f(u)=\frac{u^{\frac{1}{2}}}{2}+\frac{u^{2}}{32}$. We can show that  $0<\alpha=1/2<1=1/\eta$. Since
$f_{0}=f_{\infty}=\infty$, then {\rm (H3)} holds. Again

\begin{equation*}\label{eq-81}
\Lambda_{1}=(1-\alpha\eta)(\int_{0}^{T}(T-s)a(s)ds)^{-1}=\frac{1}{2},
\end{equation*}

because $f(u)$ is monotone increasing function for $u\geq0$, taking $\rho_{1}=4$, $M_{1}=3/8\in(0,\Lambda_{1}]$,
then when $u\in[0,\rho_{1}]$, we get

\begin{equation*}\label{eq-82}
f(u)\leq f(4)=3/2=M_{1}\rho_{1},
\end{equation*}
which implies {\rm (H4)} holds. Hence, by Theorem \ref{theo 4.1}, the BVP \eqref{eq-79}, \eqref{eq-80} has at
least two positive solutions $u_{1}$ and $u_{2}$ such that
\begin{equation*}\label{eq-83}
0<\|u_{1}\|<4<\|u_{2}\|.
\end{equation*}
\end{exmp}

\begin{exmp}
Consider the boundary value problem
\begin{equation}\label{eq-84}
{u^{\prime \prime }}(t)+8e^{3}u^{2}e^{-u}=0, \  \ 0<t<\frac{3}{4},
\end{equation}

\begin{equation}\label{eq-85}
{u^{\prime}}(0)=0, \  \ u(\frac{3}{4})=3\int_{0}^{\frac{1}{4}}u(s)ds.
\end{equation}
Set $\alpha=3$, $\eta=1/4$, $T=3/4$, $a(t)\equiv8$, $f(u)=e^{3}u^{2}e^{-u}$.
We can show that $0<\alpha=3<4=1/\eta$. Since $f_{0}=f_{\infty}=0$, then {\rm (H5)} holds. Again

\begin{eqnarray*}
\gamma&=&\frac{\alpha\eta(T-\eta)}{T-\alpha\eta^{2}}=\frac{2}{3},\\
\Lambda_{2}&=&\frac{1-\alpha\eta}{\gamma\left(\int_{\eta}^{T}(T-s)a(s)ds+\frac{1}{2}\int_{0}^{\eta}\left[2(T-\eta)+\alpha(\eta^{2}-s^{2})\right]a(s)ds\right)}\\
&=& \frac{3}{17},
\end{eqnarray*}

because $f(u)$ is monotone decreasing function for $u\geq2$, taking $\rho_{2}=3$,
$M_{2}=3\in[\Lambda_{2},\infty)$, then when $u\in[\gamma\rho_{2},\rho_{2} ]$, we obtain

\begin{equation*}\label{eq-86}
f(u)\geq f(3)=9=M_{2}\rho_{2},
\end{equation*}

which implies {\rm (H6)} holds. Hence, by Theorem \ref{theo 4.2}, the BVP \eqref{eq-84}, \eqref{eq-85} has at
least two positive solutions $u_{1}$ and $u_{2}$ such that
\begin{equation*}\label{eq-87}
0<\|u_{1}\|<3<\|u_{2}\|.
\end{equation*}
\end{exmp}

\begin{exmp}
Consider the boundary value problem
\begin{equation}\label{eq-88}
{u^{\prime \prime }}(t)+\frac{aue^{2u}}{b+e^{u}+e^{2u}}=0, \  \ 0<t<1,
\end{equation}

\begin{equation}\label{eq-89}
{u^{\prime}}(0)=0, \  \ u(1)=2\int_{0}^{\frac{1}{3}}u(s)ds,
\end{equation}
where $a=5$, $b=8$. Set $\alpha=2$, $\eta=1/3$, $T=1$, $a(t)\equiv1$, $f(u)=(aue^{2u})/(b+e^{u}+e^{2u})$. We can
show that $0<\alpha=2<3=1/\eta$. Again

\begin{eqnarray*}
\gamma&=&\frac{\alpha\eta(T-\eta)}{T-\alpha\eta^{2}}=\frac{4}{7},\\
\Lambda_{2}&=&\frac{1-\alpha\eta}{\gamma\left(\int_{\eta}^{T}(T-s)a(s)ds+\frac{1}{2}\int_{0}^{\eta}\left[2(T-\eta)+\alpha(\eta^{2}-s^{2})\right]a(s)ds\right)}\\
&=&\frac{189}{152},\\
\Lambda_{1}&=&(1-\alpha\eta)(\int_{0}^{T}(T-s)a(s)ds)^{-1}=\frac{2}{3},
\end{eqnarray*}

and

\[f_{0}=\frac{a}{b+2}=\frac{1}{2},\ \ f_{\infty}=a=5\].

Taking $\theta_{1}\in(3/4, 1]$, $\theta_{2}\in[1, 2]$, thus $f_{0}\in(0, \theta_{1}\Lambda_{1})$,
$f_{\infty}\in((\theta_{2}/\gamma)\Lambda_{2}, \infty)$, which imply {\rm (H7)} and {\rm (H8)} hold. Therefore,
by Corollary \ref{cor 5.2}, the BVP \eqref{eq-88}, \eqref{eq-89} has at least one positive solution.
\end{exmp}

\begin{exmp}
Consider the boundary value problem
\begin{equation}\label{eq-90}
{u^{\prime \prime }}(t)+\frac{1}{5}u(1+\frac{\lambda}{1+u^{2}})=0, \  \ 0<t<1,
\end{equation}

\begin{equation}\label{eq-91}
{u^{\prime}}(0)=0, \  \ u(1)=\int_{0}^{\frac{1}{2}}u(s)ds,
\end{equation}
where $\lambda\geq80$. Set $\alpha=1$, $\eta=1/2$, $T=1$, $a(t)\equiv\frac{1}{5}$,
$f(u)=u(1+\frac{\lambda}{1+u^{2}})$. We can show that $0<\alpha=1<2=1/\eta$. Again

\begin{eqnarray*}
\gamma&=&\frac{\alpha\eta(T-\eta)}{T-\alpha\eta^{2}}=\frac{1}{3},\\
\Lambda_{2}&=&\frac{1-\alpha\eta}{\gamma\left(\int_{\eta}^{T}(T-s)a(s)ds+\frac{1}{2}\int_{0}^{\eta}\left[2(T-\eta)+\alpha(\eta^{2}-s^{2})\right]a(s)ds\right)}\\
&=& 18,\\
\Lambda_{1}&=&(1-\alpha\eta)(\int_{0}^{T}(T-s)a(s)ds)^{-1}=5,
\end{eqnarray*}

and

\[f_{0}=1+\lambda,\ \ f_{\infty}=1\].

Taking $\theta_{1}\in(1/5, 1]$, $\theta_{2}\in[1, 3/2)$, thus
$f_{0}\in((\frac{\theta_{2}}{\gamma})\Lambda_{2},\infty) $, $f_{\infty}\in(0, \theta_{1}\Lambda_{1})$, which
imply {\rm (H9)} and {\rm (H10)} hold. Therefore, by Corollary \ref{cor 5.3}, the BVP \eqref{eq-90},
\eqref{eq-91} has at least one positive solution.
\end{exmp}


\begin{thebibliography}{9}



\bibitem{Ander1}
D. R. Anderson; Nonlinear triple-point problems on time scales,
  \emph{Electron. J. Diff. Eqns.,} \textbf{47} (2004), 1--12.

\bibitem{Ander2}
D. R. Anderson; Solutions to second order three-point problems on time
  scales, \emph{J. Difference Equ. Appl.,} \textbf{8} (2002), 673--688.

\bibitem{Cheng}
Z. Chengbo; Positive solutions for semi-positone three-point boundary
  value problems, \emph{J. Comput. Appl. Math.,} \textbf{228} (2009), 279--286.

\bibitem{Feng}
W. Feng, J. R. L. Webb; Solvability of a three-point nonlinear boundary value
  problem at resonance, \emph{Nonlinear Analysis TMA,} \textbf{30(6)}(1997), 3227--3238.


\bibitem{Gupt}
C.P. Gupta; Solvability of a three-point nonlinear boundary value problem
  for a second order ordinary differential equations , \emph{J. Math. Anal.
  Appl.,} \textbf{168} (1992), 540--551.

\bibitem{Guo}
Y. Guo, W. Ge; Positive solutions for three-point boundary value problems
  with dependence on the first order derivative, \emph{J. Math. Anal. Appl.,}
  \textbf{290} (2004), 291--301.


\bibitem{Han}
X. Han; Positive solutions for a three-point boundary value problem,
  \emph{Nonlinear Analysis TMA,} \textbf{66} (2007), 679--688.

\bibitem{He}
X. He, W. Ge; Triple solutions for second order three-point boundary
  value problems, \emph{J. Math. Anal. Appl.,} \textbf{268} (2002), 256--265.


\bibitem{Haddou1}
F. Haddouchi, S. Benaicha; Positive solutions of nonlinear three-point
  integral boundary value problems for second-order differential equations,
\url{[http://arxiv.org/abs/1205.1844]}, May, 2012.

\bibitem{Haddou2}
F. Haddouchi, S. Benaicha; Existence of positive solutions for a
three-point integral boundary-value problem,
\url{[http://arxiv.org/abs/1304.5644]}, April, 2013.


\bibitem{Ilin}
 V. A. Il'in, E. I. Moiseev; Nonlocal boundary-value problem of the first kind
  for a Sturm-Liouville operator in its differential and finite difference
  aspects, \emph{Differ. Equ.,} \textbf{23(7)} (1987), 803--810.


\bibitem{Krasn}
M. A. Krasnoselskii; Positive Solutions of Operator Equations, P. Noordhoff, Groningen, The Netherlands, 1964.

\bibitem{Li}
J. Li, J. Shen; Multiple positive solutions for a second-order
  three-point boundary value problem, \emph{Appl. Math. Comput.,} \textbf{182}(2006),
  258--268.


\bibitem{Luo}
H. Luo, Q. Ma; Positive solutions to a generalized second-order
  three-point boundary-value problem on time scales, \emph{Electron. J. Diff.
  Eqns.,} \textbf{17} (2005), 1--14.

\bibitem{Liang1}
S. Liang, L. Mu; Multiplicity of positive solutions for singular
  three-point boundary value problems at resonance, \emph{Nonlinear Analysis
  TMA,} \textbf{71} (2009), 2497--2505.

\bibitem{Liang2}
R. Liang, J. Peng, J. Shen; Positive solutions to a generalized second
  order three-point boundary value problem, \emph{Appl. Math. Comput.,} \textbf{196} (2008), 931--940.


\bibitem{Liu1}
B. Liu; Positive solutions of a nonlinear three-point boundary value
  problem, \emph{Appl. Math. Comput.,}  \textbf{132} (2002), 11--28.


\bibitem{Liu2}
B. Liu, L. Liu, Y. Wu; Positive solutions for singular second order
  three-point boundary value problems, \emph{Nonlinear Analysis TMA,} \textbf{66} (2007), 2756--2766.

\bibitem{Liu3}
B. Liu; Positive solutions of a nonlinear three-point boundary value
  problem, \emph{Comput. Math. Appl.,} \textbf{44} (2002), 201--211.


\bibitem{Liu4}
B. Liu, L. Liu, Y. Wu; Positive solutions for singular second order
  three-point boundary value problems, \emph{Appl. Math. Comput.,} \textbf{132} (2002), 11–-28.


\bibitem{Ma1}
R. Ma; Existence theorems for a second order three-point boundary value
  problem, \emph{J. Math. Anal. Appl.,} \textbf{212} (1997), 430--442.

\bibitem{Ma2}
R. Ma; Multiplicity of positive solutions for second-order three-point
  boundary value problems, \emph{Comput. Math. Appl.,} \textbf{40} (2000), 193--204.

 \bibitem{Ma3}
R. Ma; Positive solutions for a nonlinear three-point boundary value
  problem, \emph{Electron. J. Diff. Eqns.,}  \textbf{34} (1999), 1--8.

\bibitem{Ma4}
R. Ma; Positive solutions for second-order three-point boundary value
  problems, \emph{Appl. Math. Lett.,} \textbf{14} (2001), 1--5.


\bibitem{Ma5}
R. Ma, H. Wang; Positive solutions of nonlinear three-point boundary
  value problem, \emph{J. Math. Anal. Appl.,} \textbf{279} (2003), 216--227.

\bibitem{Maran}
S. A. Marano; A remark on a second order three-point boundary value
  problem, \emph{J. Math. Anal. Appl.,} \textbf{183}(1994), 581--522.


\bibitem{Pang}
H. Pang, M. Feng, W. Ge; Existence and monotone iteration of positive
  solutions for a three-point boundary value problem, \emph{Appl. Math. Lett.,}
  \textbf{21} (2008), 656--661.



\bibitem{Sun1}
H. R. Sun, W. T. Li; Positive solutions for nonlinear three-point boundary
  value problems on time scales, \emph{J. Math. Anal. Appl.,} \textbf{299} (2004), 508--524.

\bibitem{Sun2}
Y. Sun, L. Liu, J. Zhang, R. P. Agarwal; Positive solutions of singular
  three-point boundary value problems for second-order differential equations,
  \emph{J. Comput. Appl. Math.,} \textbf{230} (2009), 738--750.


\bibitem{Tarib}
J. Tariboon, T. Sitthiwirattham; Positive solutions of a nonlinear
  three-point integral boundary value problem, \emph{Bound. Val. Prob.,} ID 519210, doi:10.1155/2010/519210
  (2010), 11 pages.


\bibitem{Webb1} J. R. L. Webb; A unified approach to nonlocal boundary value problems,
\emph{Dynam. Systems Appl.,} 5 (2008), 510-515.

\bibitem{Webb2} J. R. L. Webb; Solutions of nonlinear equations in cones and positive
linear operators, \emph{J. Lond. Math. Soc., (2)} 82 (2010), 420-436.

\bibitem{Webb3} J. R. L. Webb, G. Infante; Positive solutions of nonlocal boundary value
problems involving integral conditions, \emph{NoDEA Nonlinear Differential Equations Appl.,} 15 (2008), 45-67.

\bibitem{Webb4} J. R. L. Webb, G. Infante; Positive solutions of nonlocal boundary value
problems: a unified approach, \emph{J. Lond. Math. Soc., (2)} 74 (2006), 673-693.

\bibitem{Webb5} J. R. L. Webb, G. Infante; Non-local boundary value problems of arbitrary order, \emph{J. Lond.
    Math. Soc., (2)} 79 (2009), 238-258.

\bibitem{Webb6} J. R. L. Webb; Positive solutions of a boundary value problem with integral boundary condition,
\emph{ Electron. J. Diff. Eqns.,} \textbf{55} (2011), 1--10.

\bibitem{Webb7} J. R. L. Webb; Nonexistence of positive solutions of nonlinear boundary value problems,
\emph{ Electron. J. Qual. Theory Differ. Equ.,} \textbf{61} (2012), 1--21.





\bibitem{Xu}
X. Xu; Multiplicity results for positive solutions of some semi-positone
  three-point boundary value problems, \emph{J. Math. Anal. Appl.,} \textbf{291} (2004), 673--689.

\end{thebibliography}
\end{document}